\documentclass[13pt,a4paper, oneside]{amsart}
\usepackage{fullpage}
\usepackage{mathrsfs}
\usepackage{CJK}
\usepackage{amsfonts}
\usepackage{xy}
\usepackage{cite}
\usepackage{amssymb}
\usepackage{color, latexsym}
\usepackage{graphicx}
\usepackage{amsbsy,amscd}
\usepackage{fancyhdr}
\usepackage[colorlinks=true,
           citecolor=blue,
           linkcolor =red]{hyperref}
\xyoption{all}
\allowdisplaybreaks

\fontsize{9pt}{9pt}\selectfont
\def\ga{\alpha}
\def\gb{\beta}
\def\gl{\lambda}
\def\bc{\mathbb{C}}
\def\hrq{H_r(q)}
\def\trace{\mathrm{trace}}

\newdimen\hoogte    \hoogte=14pt    % hoogte  van hokje
\newdimen\breedte   \breedte=14pt   % breedte van hokje
\newdimen\dikte     \dikte=0.5pt    % dikte lijn
  {\vspace{1\jot}\ifx*#2\let\pointlabel\relax\else\def\pointlabel{#2}\fi
   \refstepcounter{equation}\trivlist
   \item[\hskip\labelsep\theequation.
         \ifx\pointlabel\relax\else\space\pointlabel\space\fi]
   \ignorespaces #1
  \vspace{1\jot}}{\relax}
\numberwithin{equation}{section}

\swapnumbers
\newtheorem{theorem}[equation]{Theorem}
\newtheorem{lemma}[equation]{Lemma}
\newtheorem{proposition}[equation]{Proposition}
\newtheorem{corollary}[equation]{Corollary}
\theoremstyle{definition}
\newtheorem{definition}[equation]{Definition}
\theoremstyle{remark}
\newtheorem{remark}[equation]{Remark}

\begin{document}
\begin{CJK*}{GBK}{song}
\setlength{\itemsep}{-2\jot}
\fontsize{13}{\baselineskip}\selectfont
\setlength{\parskip}{1.0\baselineskip}
\vspace*{0mm}
\title[Characters of Iwahori--Hecke algebras]{\fontsize{9}{\baselineskip}\selectfont Characters of Iwahori--Hecke algebras}

\author{Deke Zhao}
\address{School of Applied Mathematics, Beijing Normal University at Zhuhai, Zhuhai, 519087, China}
\email{deke@amss.ac.cn}
\subjclass[2010]{Primary 20C99, 16G99; Secondary 05A99, 20C15}
\dedicatory{Dedicated to Professor Nanhua Xi on the occasion of his 55th birthday}
\keywords{Symmetric group; Iwahori--Hecke algebra; Schur--Weyl duality; Quantum superalgebra}
\vspace*{-3mm}
\begin{abstract}In this paper we prove a quantum generalization of Regev's theorems in (Israel. J. Math. 195 (2013), 31--35) by applying the Schur--Weyl duality between the quantum superalgebra and Iwahori-Hecke algebra. We also present an alternate proof of the quantized generalizations using the skew character theory of Iwahori--Hecke algebras.
\end{abstract}
\maketitle
\vspace*{-5mm}
%%%%%%%%%%%%%%%%%%%%%%%%%%%%%%%%%%%%%%%%%%%%%%%%%%%%%%%%%%%%%%%%%%%%
%%%%%%%%%%%%%%%%%%%%%%%%%%%%%%%%%%%%%%%%%%%%%%%%%%%%%%%%%%%%%%%%%%%%
%\tableofcontents
%%%%%%%%%%%%%%%%%%%%%%%%%%%%%%%%%%%%%%%%%%%%%%%%%%%%%%%%%%%%%%%%%%%%
%%%%%%%%%%%%%%%%%%%%%%%%%%%%%%%%%%%%%%%%%%%%%%%%%%%%%%%%%%%%%%%%%%%%
\section{Introduction}
Let $r$ be a positive integer and $q$ an indeterminate. The generic Iwahori--Hecke algebra $\hrq$ associated to the $r$-th symmetric group $\mathfrak{S}_r$ is the algebra over $\mathbb{C}(q)$, the field of rational functions in $q$, generated by $T_1, \ldots, T_{r-1}$ with relations
 \begin{align*}&T_i^2=(1-q)T_i+q &&\text{ for }1\leq i<r,\\
&T_iT_j=T_jT_i &&\text{ for }|i-j|>2,\\
 &T_iT_{i+1}T_i=T_{i+1}T_{i}T_{i+1} &&\text{ for }1\leq i<r-1.\end{align*}
Let $w\in \mathfrak{S}_r$ and let $s_{i_1}s_{i_2}\cdots s_{i_k}$ be a reduced expression for $w$. Then $T_{w}:=T_{i_1}T_{i_2}\cdots T_{i_k}$ is independent of the choice of reduced expression and  $\{T_{w}|w\in \mathfrak{S}_r\}$ is linear basis of $\hrq$. It should be noted that the first relation is slightly non-standard, which is related to the standard one via replacing $ T_i$ by $-T_i$. This negated version yields, in most cases, more elegant $q$-analogues (see e.g. \cite{APR}). The author is very grateful to the anonymous referee for this suggestion.

Recall that a composition (resp. partition) $\gl=(\gl_1, \gl_2, \ldots)$ of $r$, denote $\gl\models r$ (resp. $\gl\vdash r$) is a sequence (resp. weakly decreasing sequence) of  nonnegative integers such that $|\gl|=\sum_{i\geq1}\gl_i=r$ and we write $\ell(\gl)$ the length of $\gl$, i.e. the number of nonzero parts of $\gl$.   A pair $(\ga;\gb)$ of compositions is a  bicomposition of  $r$, denote $(\ga;\gb)\models r$, if $|\ga|+|\gb|=r$. \textit{A point should be pointed out that one of component of a bicomposition can be empty, i.e. $\ell(\ga)=0$ or $\ell(\beta)=0$.}

 For any $\mu=(\mu_1, \mu_2, \ldots)\vdash r$, the ``standard element" of type $\mu$ is the following element of $\hrq$:
 \begin{eqnarray*}
   T_{\gamma_\mu}&=& T_{\gamma_{\mu_1}}\times T_{\gamma_{\mu_2}}\times\cdots,
 \end{eqnarray*}
where $T_{\gamma_{\mu_i}}=T_{a_{i-1}+1}T_{a_{i-1}+2}\cdots T_{ a_i-1}$ with $a_0=0$ and $a_i=a_{i-1}+\mu_{i}$ for all $i=1, 2,\ldots$.
It is well-known that the irreducible representations of $\hrq$ are parameterized by partitions $\gl$ of $r$ (see e.g.\cite[Theorem~2.3.1]{H}). We denote $\chi^{\gl}$ the corresponding character.  Then it is known \cite[Corollary~5.2]{Ram-1991} that the characters $\chi^{\gl}$ are completely determined by their values on the ``standard elements" $T_{\gamma_{\mu}}$ for all $\mu\vdash r$. For simplicity, we write $\chi^{\gl}(\mu)=\chi^{\gl}(T_{\gamma_{\mu}})$.

Inspired by Regev's work \cite{Regev-2013} on the characters of symmetric groups,
we investigate the character $\chi_{\Phi_{m,n}^{q,r}}$ of the sign $q$-permutation representation $(\Phi_{m,n}^{q,r}, V^{\otimes r})$ of $\hrq$, where $V$ is a $\mathbb{Z}_2$-graded vector space over $\mathbb{C}(q)$ with $\dim V_{\bar 0}=m$ and $\dim V_{\bar 1}=n$, see \S\ref{Sec:Preliminaries} for details.

Our first result is an explicit formula computing the values of $\chi_{\Phi_{m,n}^{q,r}}$ on all ``standard elements" of $\hrq$ (see Theorem~\ref{Them:main}), which is a quantum version of Regev's theorem \cite[Theorem~2.2]{Regev-2013}.

\noindent\textbf{Theorem~A.} {\it Let $\mu=(\mu_1, \mu_2,\cdots)\vdash r$. Then
\begin{eqnarray*}
 \chi_{\Phi_{m,n}^{q,r}}(\mu)&=&\prod_{i=1}^{\ell(\mu)}\sum_{(\ga;\gb)\models \mu_i}\tbinom{m}{\ell(\ga)}\tbinom{n}{\ell(\gb)}(-q)^{|\gb|-\ell(\gb)}(1-q)^{\ell(\ga;\gb)-1}.
\end{eqnarray*}}

Let us remark that Theorem~A implies that \begin{eqnarray*}
 \chi_{\Phi_{m,0}^{q,r}}(\mu)&=&\prod_{i=1}^{\ell(\mu)}\sum_{(\ga)\models \mu_i}\tbinom{m}{\ell(\ga)}(1-q)^{\ell(\ga)-1},
\end{eqnarray*}
which is a negated version of the well-known result  (see e.g. \cite[Theorem~4.1 and Proposition~4.2]{Ram-1991}).

Recall that the standard combinatorial notation $[m]_q:=\frac{q^m-1}{q-1}$.
Combining the Schur-Weyl duality between the quantum superalgebra and Iwahori--Hecke algebra \cite{Moon, Mit} and  crystal basis for quantum superalgebra \cite{BKK}, the character $\chi_{\Phi_{m,n}^{q,r}}$ can be rewritten as a sum of characters of $\hrq$ labelled by the $(m,n)$-hook partitions of $r$.   This allows us to obtain the following quantum version of \cite[Proposition~1.1]{Regev-2013}.

\noindent{\bf Theorem B.} {\it Let $\mu=(\mu_1,\mu_2,\cdots)\vdash r$ and let $\chi_r:=\sum_{i=0}^{r-1}\chi^{(r-i,1^i)}$. Then
\begin{eqnarray*}
  \chi_r(\mu)&=&2^{\ell(\mu)-1}\prod_{i=1}^{\ell(\mu)}\left[\mu_i\right]_{-q}.
\end{eqnarray*}}

Very recently, Taylor \cite{Taylor} gives a new proof of Regev's work by applying the Murnaghan--Nakayama formula for the skew characters of the symmetric groups. Inspired by Taylor's work, we will provide an alternative proof of Theorem A by using the Murnaghan--Nakayama formula for the skew characters of Iwahori--Hecke algebras (Corollary~\ref{Cor:MN-formula}), which is derived from the Murnaghan--Nakayama formula for the  characters of Iwahori-Hecke algebras \cite[Theorem~3.2]{P}. As an application, an alternative proof of Theorem~B is given by using the Littlewood--Richardson rule for Iwahori--Hecke algebras (see \cite[Theorem~2]{KW} and \cite[Proposition~1.2]{GW}).

This paper is organized as follows. In Section~\ref{Sec:Preliminaries}, we  briefly review the sign $q$-permutation representation of $\hrq$ and the Schur--Weyl duality between the quantum superalgebra and Iwahori--Hecke algebra. Section~\ref{Sec:Expliit-formulas} provides the explicit formula computing the values of character of the sign $q$-permutation representation on all ``stand elements" of Iwahori--Hecke algebra. In Section~\ref{Sec:First-proof-B},  we reinterpret the character of sign $q$-permutation representation as a sum of characters of Iwahori--Hecke algebra indexed by the hook partitions and prove Theorem~B.  The last section devotes to give the alternative proofs of Theorems A and B by using the Murnaghan--Nakayama formula for skew characters and the Littlewood--Richardson rule for Iwahori--Hecke algebra.

\subsection*{Acknowledgements}Part of this work was carried out while the author was visiting the Chern Institute of Mathematics at the Nankai University and the Northeastern University at Qinhuangdao. The authors would like to thank Professor Chengming Bai and Professor Yanbo Li for their hospitalities during his visits. The author is very grateful to the anonymous referee for many valuable comments and suggestions.  The author was supported partly by the National Natural Science Foundation of China (Grant No. 11571341, 11671234, 11871107).
%%%%%%%%%%%%%%%%%%%%%%%%%%%%%%%%%%%%%%%%%%%%%%%%%%%%%%%%%%%%%%%%%%%%
%%%%%%%%%%%%%%%%%%%%%%%%%%%%%%%%%%%%%%%%%%%%%%%%%%%%%%%%%%%%%%%%%%%%

\section{Preliminaries}\label{Sec:Preliminaries}
In this section, we review briefly the sign $q$-permutation representation of the Iwahori-Hecke algebra and the Schur--Weyl duality between the quantum superalgebra and the Iwahori-Hecke algebra.

Suppose that $W=W_{\bar{0}}\oplus W_{\bar{1}}$ is a $\mathbb{Z}_2$-graded complex vector space with $\dim W_{\bar{0}}=m$ and $\dim W_{\bar{1}}=n$. For $i=\bar 0,\bar 1\in\mathbb{Z}_2$, let
\begin{align*}
  &\mathrm{End}_{\mathbb{C}}(W)_i:=\{\phi\in\mathrm{End}_{\mathbb{C}}(W)|\phi(W_j)\subseteq W_{i+j}\}.
\end{align*}
Then the general linear Lie superalgebra $\mathfrak{gl}(m,n)$ is $\mathrm{End}_{\mathbb{C}}(W)=\mathrm{End}_{\mathbb{C}}(W)_{\bar0}\oplus\mathrm{End}_{\mathbb{C}}(W)_{\bar 1}$.
Let $V=W\otimes_{\mathbb{C}}\mathbb{C}(q)$, i.e. $V$  be the $\mathbb{Z}_2$-graded vector space over $\mathbb{C}(q)$ with $\dim V_{\bar 0}=m$ and $\dim V_{\bar 1}=n$. Let $\{v_1, \cdots, v_m\}$ and $\{v_{m+1}, \cdots, v_{m+n}\}$ are
the homogeneous basis of $V$ with $V_{\bar0}=\oplus_{i=1}^m\bc(q)v_i$ and
$V_{\bar1}=\oplus_{i=m+1}^{m+n}\bc(q)v_i$. Then  $$
    \mathfrak{B}=\{v_{i_1}\otimes \cdots \otimes v_{i_{r}}|1\leq i_j\leq m+n\}$$
    is a (homogeneous) basis of $V^{\otimes r}$ for all $r\geq 1$.  For a homogeneous element $v\in V$, we denote by $|v|$ its
degree, i.e., $|v|=i$ if $v\in V_{\bar i}$.

Let $U_{q}(\mathfrak{gl}(m,n))$ be the quantized enveloping algebra of $\mathfrak{gl}(m,n)$ introduced by Benkart et al in \cite{BKK} and let $(\Psi_{m,n}, V)$ be \textit{the fundamental (super) representation} of $U_{q}(\mathfrak{gl}(m,n))$ (see \cite[\S~3.2.]{BKK}). Since $U_{q}(\mathfrak{gl}(m,n))$ is a Hopf superalgebra, the tensor product representation $(\Psi_{m,n}^{\otimes r}, V^{\otimes r})$ is a well-defined super representation of $U_{q}(\mathfrak{gl}(m,n))$ for all $r\geq 1$.

Following Moon \cite[(2.4)]{Moon} and Mitsuhashi \cite[(2)]{Mit}, we define a right operator $R$, which is a super-version of \cite[(5.2)]{APR}, on $V\otimes V$ as follows:
\begin{equation}\label{Equ:sign}
  (v_i\otimes v_{j}) R=\left\{\begin{array}{ll}\vspace{2\jot}(1-q)v_{i}\otimes v_{j}
 \sqrt{q}(-1)^{|v_i||v_j|}v_{j}\otimes v_{i}& \text{if }i<j,\\ \vspace{2\jot}
 \frac{(1-q)+(-1)^{|v_i|}(1+q)}{2} v_{i}\otimes v_{j}& \text{if }i=j, \\
  -\sqrt{q}(-1)^{|v_i||v_{j}|}qv_{j}\otimes v_{i}&\text{if } i>j.
                         \end{array}\right.
\end{equation}
For  $i=1, \cdots, r-1$, we define
\begin{eqnarray*}
  R_i&:=&\mathrm{Id}_V^{\otimes i-1}\otimes R\otimes \mathrm{Id}_V^{\otimes r-i-1}\in \mathrm{End}_{\mathbb{C}(q)}(V^{\otimes r}),
\end{eqnarray*}
where $R$ operates on the $i$th and the $(i+1)$st tensor terms.

Moon and Mitsuhashi have shown that  $\Phi_{m,n}^{q,r}: \hrq\rightarrow \mathrm{End}_{\mathbb{C}(q)}(V^{\otimes r})$ is a (super) representation of $\hrq$ given by setting $T_i\mapsto R_i$, which is called \textit{the sign $q$-permutation representation} of $\hrq$. Moreover,
Moon and Mitsuhashi established the Schur--Weyl duality between $U_{q}(\mathfrak{gl}(m,n))$ and $\hrq$ (see \cite[Theorem~3.13]{Moon}, \cite[Theorem~4.4]{Mit}).

\begin{proposition}\label{Prop:Shur-Weyl}
 Keep notations as above. Then
 \begin{align*}
   &\mathrm{End}_{U_{q}(\mathfrak{gl}(m,n))}(V^{\otimes r})=\Phi_{q,r}^{m,n}(\hrq),\\
  &\mathrm{End}_{\hrq}(V^{\otimes r})=\Psi^{\otimes r}(U_{q}(\mathfrak{gl}(m,n))).
 \end{align*}
\end{proposition}

Now we are in a position to give some remarks.
\begin{remark}i) The representation $(\Phi_{m,n}^{q,r}, V^{\otimes r})$ is exactly the classical representation of the $r$th symmetric group $\mathfrak{S}_r$ when $n=0$ and  $q=1$; The representation $(\Phi_{m,n}^{q,r}, V^{\otimes r})$ is reduced to the $q$-permutation representation of $\hrq$ introduced by Jimbo \cite{Jimbo} (see also \cite[\S5]{APR}) when $n=0$ and to the sign permutation representation of the symmetric group $\mathfrak{S}_r$ defined by Belere and Regev in \cite{B-Regev} when $q=1$.

ii) $(\Phi_{m,n}^{q,r}, \Psi_{m,n}^{\otimes r})$ is the classical (quantum) Schur-Weyl duality when $n=0$.
 \end{remark}
\section{The character of sign $q$-permutation representation}\label{Sec:Expliit-formulas}
In this section, the character $\chi_{\Phi_{m,n}^{q,r}}$ of the sign $q$-permutation representation $(\Phi_{m,n}^{q,r}, V^{\otimes r})$ of $\hrq$ is completely determined by giving an explicitly formula of $\chi_{\Phi_{m,n}^{q,r}}(\mu)$ for all ``standard elements" of $\hrq$ with type $\mu\vdash r$.

  From now on, we will identify an element $T\in\hrq$ with $\Phi_{m,n}^{q,r}(T)\in \mathrm{End}_{\mathbb{C}(q)}(V^{\otimes r})$ with respect to the basis $\mathfrak{B}$.
 Clearly, $\chi_{\Phi_{m,n}^{q,r}}$ is  completely determined by the sum of the diagonal entries in the representation $(\Phi_{m,n}^{q,r}, V^{\otimes r})$, that is, the trace of $\Phi_{m,n}^{q,r}(T_{w})$ for all ${w\in \mathfrak{S}_r}$.  Thanks to \cite[Corollary~5.2]{Ram-1991}, we only need to compute
$\trace(\Phi_{m,n}^{q,r}(T_{\gamma_\mu}))$ for all ``standard elements" $T_{\gamma_\mu}$ of $\hrq$ with type $\mu\vdash r$.
 On the other hand, owing to $\chi_{\Phi_{m,n}^{q,r}}(\mu)=
\prod_{i=1}^{\ell(\mu)}\chi_{\Phi_{m,n}^{q,r}}(\mu_i)$, it is enough to deal with $\chi_{\Phi_{m,n}^{q,r}}(T_{\vec{a}})$ for all positive integers $1< a\leq r$, where $T_{\vec{a}}=T_1\cdots T_{a-1}$.

  \begin{lemma}\label{Lem:key} Let $1<a<r$ be positive integer and let $T_{\vec{a}}=T_1\cdots T_{a-1}$. Then
   \begin{eqnarray*}
      \chi_{\Phi_{m,n}^{q,r}}(T_{\vec{a}})&=&\sum_{(\ga;\gb)\models a}\tbinom{m}{\ell(\ga)}\tbinom{n}{\ell(\gb)}(-q)^{|\gb|-\ell(\gb)}(1-q)^{\ell(\ga;\gb)-1}.
    \end{eqnarray*}
  \end{lemma}
    \begin{proof} For $\mathbf{v}\in \mathfrak{B}$, we denote by $T_{\vec{a}}\mathbf{v}|_{\mathbf{v}}$ the coefficient of $\mathbf{v}$ in the expansion of $T_{\vec{a}}(\mathbf{v})$ as a linear combination of the basis $\mathfrak{B}$.  Then
    \begin{eqnarray*}
     \chi_{\Phi_{m,n}^{q,r}}(T_{\vec{a}})&=&
     \sum_{\mathbf{v}\in\mathfrak{B}}T_{\vec{a}}\mathbf{v}|_{\mathbf{v}}.
    \end{eqnarray*}
     Given $\mathbf{v}=v_{i_1}\otimes \cdots \otimes v_{i_r}\in \mathfrak{B}$,  (\ref{Equ:sign}) implies that $T_{\vec{a}}$  acts exactly on the $a$-factor $\vec{\mathbf{v}}_{a}=v_{i_1}\otimes \cdots \otimes v_{i_{a}}$ of $\mathbf{v}$. In fact, \begin{align*}
      T_{\vec{a}}(\mathbf{v})=T_{\vec{a}}(\vec{\mathbf{v}}_a)\otimes v_{i_{a+1}}\otimes \cdots v_{i_r}.
    \end{align*}
    As a consequence, $T_{\vec{a}}\mathbf{v}|_\mathbf{v}=T_{\vec{a}}\vec{\mathbf{v}}_a|_{\vec{\mathbf{v}}_a}$.

   For a non-negative integer $i$, we write $v^{\otimes i}:=v\otimes \cdots\otimes v\in V^{\otimes i}$ when $v$ is a homogeneous element of $V$. For a bicomposition $(\alpha;\beta)$ with $\alpha=(\alpha_1,\ldots, \alpha_{\ell})$ and $\beta=(\beta_1, \ldots, \beta_k)$, we set
   \begin{eqnarray*}
   &&v^{(\alpha;\beta)}:=v_{i_1}^{\otimes \alpha_1}\otimes\cdots\otimes v_{i_{\ell}}^{\otimes \alpha_{\ell}}\otimes v_{j_1}^{\beta_1}\otimes\cdots\otimes v_{j_k}^{\otimes \beta_k},\\
   &&\mathfrak{B}_{\leq}^a=\left\{\mathbf{v}_a^{(\ga;\gb)}:=v^{(\alpha;\beta)}\left|\substack{(\ga;\beta)=(\ga_1, \cdots, \ga_{\ell}, \beta_1, \cdots, \beta_k)\models a\\1\leq i_1<\cdots<i_{\ell}\leq m \vspace{1\jot}\\ \vspace{1\jot}
     m<j_1<\cdots<j_k\leq m+n}\right.\right\}.
   \end{eqnarray*}

    Now (\ref{Equ:sign}) implies

    \begin{eqnarray*}
   T _{\vec{a}}\vec{\mathbf{v}}_a|_{\vec{\mathbf{v}}_a} &=&
    \left\{ \begin{array}{ll}
    (-q)^{|\gb|-\ell(\gb)}(1-q)^{\ell(\alpha;\beta)-1},& \hbox{if } \vec{\mathbf{v}}_a=\mathbf{v}_a^{(\ga;\gb)}\in\mathfrak{B}_{\leq}^a;\\ 0, & \hbox{ohterwise.}
    \end{array}\right.
    \end{eqnarray*}
     Note that $\dim V_{\bar 0}=m$ and $\dim V_{\bar 1}=n$, there are $\tbinom{m}{\ell(\ga)}\tbinom{n}{\ell(\gb)}$ elements being of form $\mathbf{v}_{a}^{(\ga;\gb)}$ in $\mathfrak{B}_{\leq}^{a}$ for a given bicomposition $(\ga;\gb)$ of $a$. Therefore we yield that
\begin{eqnarray*}
         \chi_{\Phi_{m,n}^{q,r}}(T_{\vec{a}})&=&
         \sum_{\mathbf{v}\in\mathfrak{B}}T_{\vec{a}}\mathbf{v}|_{\mathbf{v}}\\
     &=&\sum_{\mathbf{v}\in\mathfrak{B}_{\leq}^{a}}T_{\vec{a}}\mathbf{v}|_{\mathbf{v}}\\
     &=&\sum_{(\ga;\gb)\models a}\tbinom{m}{\ell(\ga)}\tbinom{n}{\ell(\gb)}
     (-q)^{|\gb|-\ell(\gb)}(1-q)^{\ell(\alpha;\beta)-1}.
        \end{eqnarray*}
It completes the proof.
 \end{proof}

Now we can prove the computing formula for the values of $\chi_{\Phi_{m,n}^{q,r}}$ on all ``standard elements" of $\hrq$ with type $\mu\vdash r$.
\begin{theorem}\label{Them:main} Assume that $\mu=(\mu_1, \mu_2,\cdots)\vdash r$. Then
\begin{eqnarray*}
 \chi_{\Phi_{m,n}^{q,r}}(\mu)&=&\prod_{i=1}^{\ell(\mu)}\sum_{(\ga;\gb)\models \mu_i}\tbinom{m}{\ell(\ga)}\tbinom{n}{\ell(\gb)}(-q)^{|\gb|-\ell(\gb)}(1-q)^{\ell(\ga;\gb)-1}.
\end{eqnarray*}
\end{theorem}
 \begin{proof}The proof of Lemma~\ref{Lem:key} implies that $\chi_{\Phi_{m,n}^{q,r}}(T_{\gamma_{\mu_i}})=\chi_{\Phi_{m,n}^{q,r}}(T_{\vec{\mu_i}})$ for all $i=1,2,\cdots, \ell(\mu)$. Due to $\chi_{\Phi_{m,n}^{q,r}}(\mu)=
\prod_{i=1}^{\ell(\mu)}\chi_{\Phi_{m,n}^{q,r}}(\mu_i)$,  we prove the theorem by applying Lemma~\ref{Lem:key}.
 \end{proof}

The following fact enables us to determines completely the specialization $q\rightarrow1$ of $\chi_{\Phi_{m,n}^{q,r}}(\mu)$ for all $\mu\vdash r$.

 \begin{corollary}\label{Cor:q-1}Let $\mu=(\mu_1,\mu_2,\cdots)\vdash r$. Then
 \begin{eqnarray*}
 \chi_{\Phi_{m,n}^{q,r}}(\mu)&=&\prod_{i=1}^{\ell(\mu)}
 \biggl(m+(-q)^{\mu_i-1}n\biggr)+O(1-q),
\end{eqnarray*}
where $O(1-q)$ means the remainder terms with factor $(1-q)$.
 \end{corollary}
\begin{proof}For $i=1, 2, \cdots$, Lemma~\ref{Lem:key} shows $\chi_{\Phi_{m,n}^{q,r}}(T_{\vec{\mu_i}})=\left(m+(-q)^{\mu_i-1}n\right)+O(1-q)$.
 Therefore, \begin{eqnarray*}
 \chi_{\Phi_{m,n}^{q,r}}(\mu)&=&\prod_{i=1}^{\ell(\mu)}\chi_{\Phi_{m,n}^{q,r}}(T_{\vec{\mu_i}})\\
 &=&\prod_{i=1}^{\ell(\mu)}\left[\left(m+(-q)^{\mu_i-1}n\right)+O(1-q)\right]\\
 &=&\prod_{i=1}^{\ell(\mu)}
 \biggl(m+(-q)^{\mu_i-1}n\biggr)+O(1-q).
\end{eqnarray*}
The proof is completed.
\end{proof}
We close this section with some remarks.
\begin{remark}
  If $n=0$ and $q=1$, then $(\Phi_{m,0}^{1,r}, V^{\otimes r})$ is the classical representation of $\mathfrak{S}_r$, and $\chi_{\Phi_{m,0}^{1,r}}(\mu)=m^{\ell(\mu)}$ according to Corollay~\ref{Cor:q-1}, which is a classical known result.
\end{remark}

\begin{remark}
  If $q=1$, then $\chi_{\Phi_{m,n}^{1,r}}(\mu)=\prod_{i=1}^{\ell(\mu)}\left(m-(-1)^{\mu_i}n\right)$,
  which is the Regev's theorem \cite[Theorem~2.3]{Regev-2013}.\end{remark}

  \begin{remark}
  If $n=0$ then $(\Phi_{m,0}^{q,r}, V^{\otimes r})$ is the $q$-permutation representation of $\hrq$ defined by Jimbo \cite{Jimbo} (see also \cite[\S5]{APR}), and we yield the following negated version of the well-known result  (see e.g. \cite[Theorem~4.1 and Proposition~4.2]{Ram-1991})
   \begin{eqnarray*}
    \chi_{\Phi_{m,0}^{q,r}}(\mu)&=&\prod_{i=1}^{\ell(\mu)}\sum_{\ga\models \mu_i}\tbinom{m}{\ell(\ga)}(1-q)^{\ell(\ga)-1}.
   \end{eqnarray*}
    \end{remark}

\section{Hook partition characters}\label{Sec:First-proof-B}
In this section,  combining the Schur-Weyl duality between the quantum superalgebra and Iwahori-Hecke algebra \cite{Moon,Mit} and the crystal bases theory for quantum superalgebra, we can reinterpret the  character $\chi_{\Phi_{m,n}^{q,r}}$ as a sum of irreducible characters indexed by $(m,n)$-hook partitions of $r$. As an application, we obtain a quantum version of Regev's theorem on the characters of symmetric groups \cite[Proposition~1.1]{Regev-2013}.

A partition $\gl=(\gl_1, \gl_2, \cdots)\vdash r$ is said to be \textit{an $(m, n)$-hook partition} of $r$ if $\gl_{m+1}\leq n$. We let  $H(m,n;r)$ denote the set of all $(m,n)$-hook partitions of $r$, that is
\begin{eqnarray}\label{hook.1}
H(m,n;r)=\{\gl=(\gl_1,\gl_2,\cdots)\vdash r\mid \gl_{m+1}\le n\}.
\end{eqnarray}

%Also let $H'(m,n;r)\subseteq H(m,n;r)$ denote the subset of the partitions containing the $m\times n$ rectangle: \begin{eqnarray}\label{hook.2} H'(m,n;r)=\{\gl\in H(m,n;r)\mid \gl_m\ge n\}. \end{eqnarray}

It is known that $(\Psi_{m,n}^{\otimes r}, V^{\otimes r})$ are completely reducible representations of $U_{q}(\mathfrak{gl}(m,n))$ for all $r\geq 1$ \cite[Proposition~3.1]{BKK}. Furthermore, Moon \cite[Theorem~5.16]{Moon} and Mitsuhashi \cite[Theorem~5.1]{Mit} have obtained the following decomposition of $\hrq\otimes  U_{q}(\mathfrak{gl}(m,n))$-modules:
\begin{eqnarray}\label{Equ:Decom}
V^{\otimes r}&=&\bigoplus_{\gl\in H(m,n;r)}H_{\gl}\otimes V_{\gl},
\end{eqnarray}
where $H_{\gl}$ is the irreducible $\hrq$-module labelled by $\gl$ and $V_{\gl}$ is the irreducible
$U_{q}(\mathfrak{gl}(m,n))$-module with highest weight $\gl$ such that $V_{\gl}\ncong V_{\mu}$ if $\gl\neq \mu$.

Recall that a tableau is called {\it semi-standard} if it is weakly increasing in rows and
strictly in columns; if its entries are from the set $\{1,2,\ldots,r\}$ then it is called
\textit{an $r$ tableau}. Of course an  $r$ tableau is also an $r+1$ tableau, etc.
\begin{definition}Let $m, n$ be nonnegative integers with $m+n>0$ and let $\gl\vdash r$. Let $\bar{\mathbf{0}}=\{0_1, \cdots, 0_m\}$ and $\bar{\mathbf{1}}=\{1_1, \cdots, 1_n\}$ with $0_1<\cdots<0_m<1_1<\cdots<1_n$. Then a tableau $\mathcal{T}$ of shape $\gl$ is said to be $(m,n)$-\textit{semistandard} if
\begin{enumerate}
  \item[(i)] the $\bar{\mathbf{0}}$ part (i.e. the boxes filled with entries $0_i$'s) of $\mathcal{T}$ is a tableau,
  \item[(ii)]the $0_i$'s are nondecreasing in row, strictly increasing in columns,

   \item[(ii)]the $1_i$'s are nondecreasing in columns, strictly increasing in rows.
\end{enumerate}
 \end{definition}
We denote by $s_{m,n}(\gl)$ the number of $(m,n)$-semistandard tableaux of shape $\gl$. Then $s_{m,n}(\gl)\neq 0$ if and only if $\gl$ is an $(m,n)$-hook partition (see \cite[\S\,2]{B-Regev}, \cite[Lemma~4.2]{BKK}).

Benkart et al \cite{BKK} have shown that the irreducible summands of the representation $(\Psi_{m,n}^{\otimes r}, V^{\otimes r})$ of $U_q(\mathfrak{gl}(m,n))$ are indexed by the $(m,n)$-hook partitions of $r$. Furthermore, Given $\gl\in H(m,n;r)$, the irreducible summand $V_{\gl}$ labelled by $\gl$ has a basis parameterized by the  $(m,n)$-semistandard tableaux of shape $\gl$, which means $\dim _{\mathbb{C}(q)}V(\gl)=s_{m,n}(\gl)$.

Combining Proposition~\ref{Prop:Shur-Weyl} and Equ.~\ref{Equ:Decom}, we get that  \begin{eqnarray*}
                         V^{\otimes r}&\cong& \bigoplus_{\gl\in H(m,n;r)} s_{m,n}(\gl)H_{\gl}
                           \end{eqnarray*}
as a right $\hrq$-module,  which gives us the following interpretation of the character of the sign $q$-permutation representation of Iwahori-Hecke algebra:
 \begin{eqnarray}\label{Equ:hook}
   \chi_{\Phi_{m,n}^{q,r}}&=&\sum_{\gl\in H(m,n;r)}s_{m,n}(\gl)\chi^{\gl}.
 \end{eqnarray}

 \begin{corollary}\label{Cor:hook} Assume that $\mu=(\mu_1, \mu_2,\cdots)\vdash r$. Then
\begin{eqnarray*}
 \sum_{\gl\in H(m,n;r)}s_{m,n}(\gl)\chi^\gl(\mu)&=&\prod_{i=1}^{\ell(\mu)}\sum_{(\ga;\gb)\models \mu_i}\tbinom{m}{\ell(\ga)}\tbinom{n}{\ell(\gb)}
 (-q)^{|\beta|-\ell(\beta)}(1-q)^{\ell(\ga;\gb)-1}.
\end{eqnarray*}
\end{corollary}

\begin{proof}
  It follows directly by applying Theorem~\ref{Them:main} and Equ.~(\ref{Equ:hook}).
\end{proof}

Recall that $[m]_q:=\frac{q^m-1}{q-1}$. Now we can prove Theorem~B stated in Introduction.
\begin{theorem}\label{Them:II} Let $\mu=(\mu_1, \mu_2, \cdots)\vdash r$ and let
  $\chi_r:=\sum_{i=0}^{r-1}\chi^{(r-i,1^i)}$. Then
\begin{eqnarray*}
  \chi_r(\mu)&=&2^{\ell(\mu)-1}\prod_{i=1}^{\ell(\mu_i)}\left[\mu_i\right]_{-q}.
\end{eqnarray*}
\end{theorem}
\begin{proof}
If $m=n=1$, then $H(1,1;r)=\{(r-i, 1^i)|i=0, 1, \cdots, r-1\}$ and  $s_{1,1}(\gl)=2$ for all $\gl\in H(1,1;n)$ due to \cite[Theorem~6.24]{B-Regev}.  Applying Corollary~\ref{Cor:hook}, we obtain
\begin{eqnarray*}
\chi_r(\mu) &=& \frac{1}{2}\chi_{\Phi_{1,1}^{q,r}}(\mu)\\
&=&\frac{1}{2}\prod_{i=1}^{\ell(\mu)}\sum_{(\ga;\gb)\models \mu_i}\tbinom{1}{\ell(\ga)}\tbinom{1}{\ell(\gb)}(-q)^{|\beta|-\ell(\beta)}(1-q)^{\ell(\ga;\gb)-1}\\
  &=&\frac{1}{2}\prod_{i=1}^{\ell(\mu)}\biggl(1+
  (-q)^{\mu_i-1}+(1-q)\sum_{\ga+\gb=\mu_i; \ga,\gb>0}(-q)^{\gb-1}\biggr)\\
&=&2^{\ell(\mu)-1}\prod_{i=1}^{\ell(\mu)}\frac{1-(-q)^{\mu_i}}{1+q}\\
&=&2^{\ell(\mu)-1}\prod_{i=1}^{\ell(\mu)}\left[\mu_i\right]_{-q}.
\end{eqnarray*}
The proof is completed.
\end{proof}

\begin{remark}
  If $q=1$, then \begin{align*}
  \sum_{i=0}^{r-1}\chi^{(r-i, 1^i)}(\mu)&=\frac{1}{2}\sum_{j=1}^{\ell(\mu)}(1-(-1)^{\mu_j})\\
&=
\left\{                                                                                             \begin{array}{ll}                                                                                               2^{\ell(\mu)-1} & \hbox{if }|\mu_j| \hbox{ is odd for all }j,  \\                                                                                               0& \hbox{otherwise,}\end{array}\right.
\end{align*}
which is the main result of \cite[Proposition~1.1]{Regev-2013}.
\end{remark}

For $i=1, \cdots, \ell(\mu)$, we denote by
\begin{eqnarray*}
                \Theta_{m,n}^{k,\ell}(\mu_i) &:=& \sum_{\substack{(\ga;\gb)\models \mu_i;\ell(\ga)=k,\ell(\gb)=\ell}}\tbinom{m}{k}\tbinom{n}{\ell}
 (-q)^{|\beta|-\ell}(1-q)^{k+\ell-1}.
               \end{eqnarray*}
Then $\Theta_{m,n}^{1,0}(\mu_i)=m$, $\Theta_{m,n}^{0,1}(\mu_i)=n(-q)^{\mu_i-1}$, $\Theta_{m,n}^{2,0}(\mu_i)=\tbinom{m}{2}(\mu_i-1)(1-q)$,
\begin{eqnarray*}
 \Theta_{m,n}^{1,1}(\mu_i) &=& mn(1-q)\sum_{\substack{\ga+\gb=\mu_i;\ga,\gb>0}}(-q)^{\beta-1}\\
  &=&mn\left(1-(-q)^{\mu_i-1}\right)\frac{1-q}{1+q},\\
  \Theta_{m,n}^{2,1}(\mu_i) &=& n\tbinom{m}{2}(1-q)^2
 \sum_{\substack{\ga_1+\ga_2+\gb=\mu_i\\\ga_1,\ga_2,\gb>0}}(-q)^{\beta-1}\\
 &=&n\tbinom{m}{2}(1-q)^2\sum_{\gb=1}^{\mu_i-2}(\mu_i-1-\beta)(-q)^{\beta-1}\\
  &=&n\tbinom{m}{2}\biggl((\mu_i-1)(1+q)-(1-(-q)^{\mu_i-1}\biggr)\frac{(1-q)^2}{(1+q)^2}.
 \end{eqnarray*}

 The following specialization of Theorem~A was suggested to the author by the referee.
 \begin{corollary}\label{Cor:2-1-phi}Let $\mu=(\mu_1, \mu_2,\ldots)\vdash r$.
  Then \begin{eqnarray*}
\chi_{\Phi_{2,1}^{q,r}}(\mu)  &=&\prod_{i=1}^{\ell(\mu)}\biggl(\frac{2+6q+4(-q)^{\mu_i+1}}{(1+q)^2}+(2\mu_i-1)\frac{1-q}{1+q}\biggr).
\end{eqnarray*}
\end{corollary}
\begin{proof} For $(m,n)=(2,1)$, Theorem~A implies that
  \begin{eqnarray*}
\chi_{\Phi_{2,1}^{q,r}}(\mu)
 &=&\prod_{i=1}^{\ell(\mu)}\left(\Theta_{2,1}^{1,0}(\mu_i)+ \Theta_{2,1}^{0,1}(\mu_i)+
    \Theta_{2,1}^{2,0}(\mu_i)+\Theta_{2,1}^{1,2}(\mu_i)+\Theta_{2,1}^{2,1}(\mu_i)\right)\\
 &=&\prod_{i=1}^{\ell(\mu)}\biggl(\frac{2+6q+4(-q)^{\mu_i+1}}{(1+q)^2}+(2\mu_i-1)\frac{1-q}{1+q}\biggr).   \end{eqnarray*}
 The proof is completed.
\end{proof}

Now we can state the following $q$-analogue of \cite[Corollary~3.1]{Regev-2013}.

\begin{corollary}
   Denote by $H'(2,1;r)\subseteq H(2,1;r)$ the $(2,1)$-hook partitions $\lambda=(\lambda_1,\lambda_2, \ldots)$ of $r$ with $\lambda_2>0$ and let $\mu=(\mu_1, \mu_2,\ldots)\vdash r$. Then
\begin{align*}
 \sum_{\gl\in H(2,1;r)}(\lambda_1\negmedspace-\negmedspace\lambda_2\negmedspace+\negmedspace1)\chi^\gl(\mu)=\frac{1}{4}\biggl(\prod_{i=1}^{\ell(\mu)}
 \left(\frac{2\negmedspace+\negmedspace6q\negmedspace+\negmedspace4(-q)^{\mu_i+1}}{(1+q)^2}
 \negmedspace+\negmedspace(2\mu_i\negmedspace-\negmedspace1)\frac{1-q}{1+q}\right)\negmedspace-
 \negmedspace(2r\negmedspace+\negmedspace1)\biggr).
\end{align*}
\end{corollary}

\begin{proof}Note that $\lambda=(r)$ is the only partition belonging to $H(2,1;r)$, which is not being in $H'(2,1;r)$. It follows from \cite[Definition~2.1]{B-Regev} that $s_{2,1}(r)=2r+1$. For $\lambda\in H'(2,1;r)$, \cite[Theorem~6.24]{B-Regev} implies that $s_{2,1}(\lambda)=4(\lambda_1-\lambda_2+1)$.  Thanks to the negated version of \cite[Theorem~6.7 and (8.2)]{Ram-1991}, Equ.~\ref{Equ:hook} and Corollary~\ref{Cor:2-1-phi}, we obtain that
 \begin{eqnarray*}
 \sum_{\gl\in H'(2,1;r)}(\negmedspace\lambda_1\negmedspace-\negmedspace\lambda_2
 \negmedspace+\negmedspace1)\chi^\gl(\mu)&=&
 \frac{1}{4}\biggl(\chi_{\Phi_{2,1}^{q,r}}(\mu)-(2r+1)\chi^{(r)}(\mu)\biggr)\\
 &=&\frac{1}{4}\biggl(\prod_{i=1}^{\ell(\mu)}
 \left(\frac{2\negmedspace+\negmedspace6q\negmedspace+\negmedspace4(-q)^{\mu_i+1}}{(1+q)^2}
 \negmedspace+\negmedspace(2\mu_i\negmedspace-\negmedspace1)\frac{1\negmedspace-\negmedspace q}{1\negmedspace+\negmedspace q}\right)\negmedspace-
 \negmedspace(2r\negmedspace+\negmedspace1)\negmedspace\biggr).
 \end{eqnarray*}
It completes the proof.
\end{proof}

\section{Alternative Proofs of Theorems A and B}\label{Sec:second-proof-A-B}

 This section devotes to give an alternative proof of Theorem~A by applying the Murnaghan--Nakayama formula for skew characters of Iwahori--Hecke algebras, which is inspired by Taylor's  work \cite{Taylor}. Combing Theorem~A and the Littlewood--Richardson rule, we obtain a new proof of Theorem~B.

Recall that the \textit{Young diagram} of a partition $\pi$ may be formally defined as the set
\begin{eqnarray*}
 [\pi] &:=& \left\{(i,j)\,\mid\, i\geq 1, 1 \leq j \leq \pi_i\right\}.
\end{eqnarray*}
 We may and will identify a partition $\pi$ with its Young diagram.
Assume that $\rho,\pi$ are partitions with $\rho\subseteq \pi$, i.e., $\ell(\rho)\leq\ell(\pi)$  and $\rho_i\leq \pi_i$, $i=1, \ldots, \ell(\pi)$. The set-theoretic difference $\pi\backslash\rho$ is called a \textit{skew diagram}. For a connected component of the skew diagram $\pi\backslash\rho$, we define its the \textit{length} to be one less than the number of rows it occupies and define the \textit{length}  $\ell(\pi\backslash\rho)$ of the skew diagram $\pi\backslash\rho$ to be the sum of the length of its connected components.  The skew diagram $\pi\backslash\rho$ is called a \textit{strip} if it does not contain any $2\times 2$ block and we denote by $c(\pi\backslash\rho)$ the number of connected components of $\pi\backslash\rho$. A connected strip is said to \textit{horizontal} (resp.  \textit{vertical}) if it has at most one box in each column (resp. row).

 For any partition $\pi$, the {\it rim} $R(\pi)$ in the diagram $\pi$ is \begin{eqnarray*}
 R(\pi)&:=& \bigcup_{i=1}^{\ell(\pi)}\left\{(i,j)\,\mid\, \pi_{i+1}\leq j\leq \pi_i\right\},
\end{eqnarray*}
where $\pi_{\ell(\pi)+1}=0$ and define the rim  of the skew diagram $\pi\backslash\rho$ to be \begin{eqnarray*}
 R(\pi\backslash\rho)&:=& R(\pi)\cap(\pi\backslash\rho).
\end{eqnarray*}
A \textit{rim strip} of $\pi\backslash\rho$ is a strip $S\subseteq R(\pi\backslash\rho)$ such that for any $(x,y)\in S$ and any $(a,b)\in R(\pi\backslash\rho)\backslash S$, either $a<x$ or $b<y$. We say that $S$ is a $k$\textit{-rim strip} of $\pi\backslash\rho$ if $|S|=k$ and denote by $S_k(\pi\backslash\rho)$ the set of all $k$-rim strips of $\pi\backslash\rho$.

\begin{remark}
  A rim strip $S$ has the property that $(\pi\backslash\rho)\backslash S$ is again a skew diagram. Indeed, we have $(\pi\backslash\rho)\backslash S=(\pi\backslash S)\backslash \rho$.
\end{remark}

 Given any positive integer $k$, we assume that $\rho\vdash k$ and $\pi\vdash r+k$ such that $\rho\subset \pi$. We denote by  $\chi^{\pi\backslash\rho}$  the skew character of $\hrq$ associated to the skew diagram $\pi\backslash\rho$.
Now we can state the Murnaghan--Nakayama formula for the skew character of the negated version of Iwahori--Hecke algebra $\hrq$, which is derived from \cite[Theorem~3.2]{P}) by applying the similar argument as that of \cite[Chapter~2.4]{JK}.

\begin{theorem}\label{Them:MN-formula} Let $\mu$ be a partition of $r$ and  denote by $\hat{\mu}$ the partition of $r-\mu_i$ obtained by deleting the $i$-component $\mu_i$ of $\mu$ for some $1\leq i\leq \ell(\mu)$. If $\rho\vdash \mu_i$ and $\pi\vdash r+\mu_i$ such that $\rho\subset \pi$, then
\begin{eqnarray*}
  \chi^{\pi\backslash\rho}(\mu) &=&
  \sum_{\substack{S\in S_{\mu_i}(\pi\backslash\rho)}}
   (-q)^{\mu_i-\ell(S)-c(S)}(1-q)^{c(S)-1}
  \chi^{(\pi\backslash\rho)\backslash S}(\hat{\mu}).
\end{eqnarray*}
\end{theorem}

Let $(\alpha;\beta)$ be a bicomposition $r$. We denote by $S{(\alpha;\beta)}$ the skew diagram whose connected components are horizontal strips $H_1, \cdots, H_m$ and vertical strips $V_1, \cdots, V_{n}$, and $(|H_1|, \cdots, |H_m|; |V_1|,\cdots, |V_n|)=(\alpha;\beta)$. Associated to $S{(\alpha;\beta)}$, we define the corresponding character of $\hrq$ \begin{eqnarray}\label{Equ:skew-character}
  \chi^{(\alpha;\beta)} &:=&\chi^{(\alpha_1)}\hat{\otimes}\cdots\hat{\otimes}\chi^{(\alpha_{m})}
  \hat{\otimes}\chi^{(1^{\beta_1})}\cdots\hat{\otimes}\chi^{(1^{\beta_{n}})},
\end{eqnarray}
where $\hat{\otimes}$ denotes the outer product of characters.

Note that for any positive integer $k$, the set of $k$-strips of the skew diagram $S(\alpha;\beta)$ is \begin{eqnarray*}
                S_k(\alpha;\beta) &=& \{S(\gamma,\delta)\,\mid\,(\alpha;\beta)\supset(\gamma,\delta)\models k\}.
                                         \end{eqnarray*}
 Clearly, $\ell(S)=|\delta|-\ell(\delta)$, $c(S)=\ell(\gamma;\delta)$. Therefore
 Theorem~\ref{Them:MN-formula} implies the following fact, which is crucial to the alternative proof of Theorem~A.

\begin{corollary}\label{Cor:MN-formula} Let $(\alpha;\beta)=(\alpha_1, \cdots, \alpha_m;\beta_1, \cdots, \beta_n)\models r$ and denote by $\hat{\mu}$ the partition of $r-\mu_i$ obtained by deleting the $i$-component $\mu_i$ of $\mu$ for some $1\leq i\leq \ell(\mu)$. Then
\begin{eqnarray*}
  \chi^{(\alpha;\beta)}(\mu) &=&
  \sum_{\substack{(\alpha;\beta)\supset(\gamma;\delta)\models \mu_i}}
  (-q)^{|\delta|-\ell(\delta)}(1-q)^{\ell(\gamma;\delta)-1}
  \chi^{(\alpha;\beta)\backslash(\gamma;\delta)}(\hat{\mu}).
\end{eqnarray*}
\end{corollary}

  For $\lambda\vdash r$ and $(\alpha;\beta)\models r$, we say that a $(m,n)$-semistandard tableau $\mathfrak{t}$ of shape $\lambda$ is of \textit{weight} $(\alpha;\beta)$ if the numbers of nodes of $\lambda$ with entries $0_i$ ($1\leq i\leq m$) and $1_j$ ($1\leq j\leq n$) are $\alpha_i$ and $\beta_j$ respectively. In other words, $\mathfrak{t}$  can be viewed as a sequence of partitions
 \begin{align*}
   \rho=\pi^{(0)}\subseteq\pi^{(1)}\subseteq\ldots\subseteq\pi^{(m+n)}=\pi
 \end{align*}
 such that the skew diagrams $\theta^{(i)}=\pi^{(i)}\backslash\pi^{(i-1)}$ ($1\leq i\leq m+n$)
is either a horizontal strip or a vertical strip, and the horizontal strips (resp. the vertical strips) are of the form $\alpha$ (resp. $\beta$). Let $s_{(\alpha;\beta)}(\lambda)$ be the number of semistandard $(m,n)$-tableaux of shape $\lambda$ with weight $(\alpha;\beta)$. Then $s_{m,n}(\lambda)=\sum_{(\alpha;\beta)\models r}s_{(\alpha;\beta)}(\lambda)$.

The following fact establishes the relationship between the skew characters and the irreducible characters of $\hrq$, which follows the proof of \cite[Lemma~3.2]{Taylor} and that $\hrq$ is generic.
\begin{corollary}\label{Cor:skew-hook}
  Assume that $(\alpha;\beta)$ is a bicomposition of $r$ with $\alpha=(\alpha_1, \ldots, \alpha_m)$ and $\beta=(\beta_1, \ldots, \beta_n)$.  Then
   \begin{eqnarray*}
     \chi^{S(\alpha;\beta)} &=&\sum_{\lambda\in H(m,n;r)}s_{(\alpha;\beta)}(\lambda)\chi^{\lambda}.
   \end{eqnarray*}
 \end{corollary}

%\begin{lemma}\label{Lem:skew-hook}For any nonnegative integers $m,n,r$, we have \begin{eqnarray*}
% \nonumber to remove numbering (before each equation)  \chi_{\Phi_{m,n}^{q,r}} &=&\sum_{(\alpha;\beta)\models r}\psi_{S_{(\alpha;\beta)}}.\end{eqnarray*}  \end{lemma}

 Now we can give an alternative proof of Theorem~A as follow.

\begin{proof}[Proof of Theorem~A]
   We prove the theorem by the induction argument on the length $\ell(\mu)$ of $\mu$.
   At first, consider the skew diagram $S(\alpha;\beta)$ with weight $(\alpha;\beta)\models r$, where $\alpha=(\alpha_1, \ldots, \alpha_m)$ and $\beta=(\beta_1, \ldots, \beta_n)$.
   If $S(\gamma;\delta)\subseteq S(\alpha;\beta)$ is a skew diagram with weight $(\gamma;\delta)\models\mu_i$ for some $1\leq i\leq \ell(\mu)$, then due to Equ.~\ref{Equ:skew-character} and Corollary~\ref{Them:MN-formula}, we yield that
   \begin{eqnarray*}
        \chi^{S(\alpha;\beta)}(\mu) &=& \sum_{\substack{(\alpha;\beta)\supset(\gamma;\delta)\models\mu_i}}
     (-q)^{|\delta|-\ell(\delta)} (1-q)^{\ell(\gamma;\delta)-1}\chi^{S(\alpha;\beta)\backslash(\gamma;\delta)}(\hat{\mu}).
   \end{eqnarray*}

    Note that for  $1\leq i\leq\ell(\mu)$, any bicomposition $(\alpha;\beta)\models r$ may be written to be the following form
    \begin{eqnarray}\label{Equ:composition+}
      (\alpha'+\delta;\beta'+\gamma)&:=&(\alpha'_1+\delta_1, \ldots, \alpha_m'+\delta_m; \beta_1'+\gamma_1,\ldots,\beta_n'+\gamma_n)
    \end{eqnarray}
   where $(\alpha';\beta')$ and $(\delta;\gamma)$ are bicompositions of $r-\mu_i$ and of $\mu_i$ respetively.  Applying Corollary~\ref{Cor:skew-hook}, we obtain that
   \begin{eqnarray*}\chi_{\Phi_{m,n}^{q,r}}(\mu) &=&\sum_{\lambda\vdash H(m,n;r)}s_{m,n}(\lambda)\chi^{\lambda}(\mu)\\
          &=&\sum_{\lambda\vdash H(m,n;r)}\sum_{(\alpha;\beta)\models r}s_{(\alpha;\beta)}(\lambda)\chi^{\lambda}(\mu)\\
          &=&\sum_{(\alpha;\beta)\models r}\chi^{(\alpha;\beta)}(\mu)\\
          &=&\sum_{(\alpha;\beta)\models r}\sum_{\substack{(\alpha;\beta)
             \supset(\gamma;\delta)\models\mu_i}}
                   (-q)^{|\delta|-\ell(\delta)} (1-q)^{\ell(\gamma;\delta)-1}
         \chi^{(\alpha;\beta)\backslash(\gamma;\delta)}(\hat{\mu}) \\
         &=&\sum_{\substack{(\gamma;\delta)\models\mu_i}}
         \tbinom{m}{\ell(\gamma)}\tbinom{n}{\ell(\delta)} (-q)^{|\delta|-\ell(\delta)}(1-q)^{\ell(\gamma;\delta)-1}
         \sum_{(\alpha;\beta)\models r-\mu_i} \chi^{(\alpha;\beta)}(\hat{\mu})\\
  &=&\chi_{\Phi_{m,n}^{q,r-\mu_i}}(\hat{\mu})\sum_{\substack{(\gamma;\delta)\models\mu_i}}
    \tbinom{m}{\ell(\gamma)}\tbinom{n}{\ell(\delta)}
  (-q)^{|\delta|-\ell(\delta)}(1-q)^{\ell(\gamma;\delta)-1},
        \end{eqnarray*}
    where the fifth equality follows the fact that given a bicomposition $(\alpha;\beta)$ of $r$ and a bicomposition $(\alpha';\beta')$ of $r-\mu_i$, there exactly $\tbinom{m}{\ell(\gamma)}\tbinom{n}{\ell(\delta)}$ bicompositions $(\delta;\gamma)$ of $\mu_i$ with given lengths satisfying Equ.~\ref{Equ:composition+}.
 \end{proof}

We are now ready to give an alternative proof of Theorem~B.

 \begin{proof}[Proof of Theorem~B] Now assume that $m=n=1$. Then the skew characters occurring in $\Theta_{1,1}^{q,r}$ are of the form $\chi^{(a;r-a)}$ with $0\leq a\leq r$. Without loss of generality, we may assume that $\chi^{(a;r-a)}=\chi^{\alpha\backslash\beta}$ where $\alpha=(a+1,1^{r-a})\vdash r+1$ and $\beta=(1)\vdash1$. Applying the Littlewood--Richardson rule (see e.g. \cite[Theorem~2]{KW}),
  \begin{eqnarray*}
    \chi^{(a;n-a)} &=&\sum_{\gamma\vdash r}c_{\beta\gamma}^{\alpha}\chi^{\gamma},
  \end{eqnarray*}
  where $c_{\beta\gamma}^{\alpha}$ is exactly the usual Littlewood--Richardson coefficient thanks to \cite[Proposition~1.2]{GW}. Now for $\beta=(1)$ the Littlewood--Richardson coefficient is described by the branching rule (cf. \cite[Theorem~2.4.3]{JK}). Applying this rule we easily deduce that
  \begin{eqnarray}\label{Equ:hook-part-char}
    \chi^{(a;r-a)} &=& \left\{
                           \begin{array}{ll}
                             \chi^{(1^r)}, & \hbox{if }a=0, \\
                            \chi^{(a,1^{r-a})}+\chi^{(a+1,1^{r-a-1})}, & \hbox{if }0<a<r, \\
                             \chi^{(r)}, & \hbox{if }a=r.
                           \end{array}
                         \right.
  \end{eqnarray}
  Recall that $\chi_r=\sum_{i=0}^{r-1}\chi^{(r-i,1^i)}$. Equ.~\ref{Equ:hook-part-char}  and Theorem~A imply
  \begin{eqnarray*}
    \chi_r(\mu) &=&\frac{1}{2}\chi_{\Phi_{1,1}^{q,r}}(\mu)\\
    &=&\frac{1}{2}\prod_{i=1}^{\ell(\mu)}\sum_{(\alpha;\beta)\models \mu_i}\tbinom{1}{\ell(\alpha)}\tbinom{1}{\ell(\beta)}
 (-q)^{|\beta|-\ell(\beta)}(1-q)^{\ell(\alpha;\beta)-1}\\
  &=&\frac{1}{2}\prod_{i=1}^{\ell(\mu)}\biggl(1+
  (-q)^{\mu_i-1}+(1-q)\sum_{\ga+\gb=\mu_i; \ga,\gb>0}(-q)^{\gb-1}\biggr)\\
&=&2^{\ell(\mu)-1}\prod_{i=1}^{\ell(\mu)}\frac{1-(-q)^{\mu_i}}{1+q}\\
&=&2^{\ell(\mu)-1}\prod_{i=1}^{\ell(\mu)}\left[\mu_i\right]_{-q},
  \end{eqnarray*}
  which  proves Theorem~B stated in Introduction.\end{proof}

%%%%%%%%%%%%%%%%%%%%%%%%%%%%%%%%%%%%%%%%%%%%%%%%%%%%%%%%%%%%%%%%%%%%%%
%\bibliographystyle{amsplain}

\end{CJK*}
\end{document}